\documentclass[10pt]{amsart}
\usepackage[margin=1.2in,marginparsep=0.1in,marginparwidth=1in]{geometry}
\usepackage{amssymb,amsmath,amsthm,amstext,amscd,latexsym,graphics,graphicx,bbm,caption}
\usepackage[usenames,dvipsnames,svgnames,table]{xcolor}
\usepackage[plainpages=false,colorlinks=true, pagebackref]{hyperref}
\usepackage{tikz}
\usetikzlibrary{cd,positioning}

\tikzset{>=stealth}
\hypersetup{citecolor=Sepia,linkcolor=blue, urlcolor=blue}

\usepackage{cite}   

\makeatletter
\def\@tocline#1#2#3#4#5#6#7{\relax
  \ifnum #1>\c@tocdepth 
  \else
    \par \addpenalty\@secpenalty\addvspace{#2}%
    \begingroup \hyphenpenalty\@M
    \@ifempty{#4}{%
      \@tempdima\csname r@tocindent\number#1\endcsname\relax
    }{%
      \@tempdima#4\relax
    }%
    \parindent\z@ \leftskip#3\relax \advance\leftskip\@tempdima\relax
    \rightskip\@pnumwidth plus4em \parfillskip-\@pnumwidth
    #5\leavevmode\hskip-\@tempdima
      \ifcase #1
       \or\or \hskip 2em \or \hskip 2em \else \hskip 3em \fi%
      #6\nobreak\relax
    \dotfill\hbox to\@pnumwidth{\@tocpagenum{#7}}\par
    \nobreak
    \endgroup
  \fi}
\makeatother

\newtheorem{intro-thm}{Theorem}[]
\theoremstyle{plain}
\newtheorem{thm}{Theorem}[section]
\newtheorem{theorem}[thm]{Theorem}
\newtheorem{q}[thm]{Question}
\newtheorem{lemma}[thm]{Lemma}
\newtheorem{corollary}[thm]{Corollary}
\newtheorem{proposition}[thm]{Proposition}

\theoremstyle{definition}
\newtheorem{remark}[thm]{Remark}

\newtheorem{definition}[thm]{Definition}
\newtheorem{example}[thm]{Example}

\newcommand{\Proj}{{\P roj}}

\newcommand{\codim}{{\rm codim}}

\newcommand{\Spec}{{\rm Spec \,}}

\newcommand{\sE}{{\mathcal E}}

\newcommand{\sG}{{\mathcal G}}
\newcommand{\sH}{{\mathcal H}}

\newcommand{\sO}{{\mathcal O}}

\newcommand{\sU}{{\mathcal U}}
\newcommand{\sV}{{\mathcal V}}
\newcommand{\sW}{{\mathcal W}}
\newcommand{\sX}{{\mathcal X}}
\newcommand{\sY}{{\mathcal Y}}
\newcommand{\sZ}{{\mathcal Z}}
\newcommand{\A}{{\mathbb A}}

\renewcommand{\P}{{\mathbb P}}
\newcommand{\Q}{{\mathbb Q}}

\newcommand{\Z}{{\mathbb Z}}

\newcommand{\colim}{{\rm colim \,}}
\newcommand{\hocolim}{{\rm hocolim \,}}
\newcommand{\holim}{{\rm holim \,}}
\newcommand{\DM}[2]{\mathbf{DM}_{#2}^{\mathit{eff}}(#1)}

\newcommand{\fixme}[1]{\textcolor{red}{$\bullet$}\marginpar{\small \textcolor{red}{#1}}}

\usepackage{verbatim}

\input{xy}
\xyoption{all}
\begin{document}

\title{The Nisnevich Motive of an Algebraic Stack}
\author{Utsav Choudhury}
\address{Stat Math unit, Indian Statistical Institute Kolkata, 203 Barrackpore Trunk Road, Kolkata 700108 India}
\email{prabrishik@gmail.com}

\author{Neeraj Deshmukh}
\address{Institut f\"{u}r Mathematik, Universit\"{a}t Z\"{u}rich, Winterthurerstrasse 190, CH-8057 Z\"{u}rich, Switzerland}
\email{neeraj.deshmukh@math.uzh.ch}

\author{Amit Hogadi}
\address{Department of Mathematics, Indian Insititute of Science Education and Research (IISER) Pune, Dr. Homi Bhabha road, Pashan, Pune 411008 India}
\email{amit@iiserpune.ac.in}



\date{}

\begin{abstract}
We construct the motive of an algebraic stack in the Nisnevich topology. For stacks which are Nisnevich locally quotient stacks, we give a presentation of the motive in terms of simplicial schemes. We also show that the motivic cohomology agrees with the Chow groups of Edidin-Graham-Totaro with integer coefficients. 
\end{abstract}

\maketitle

\section{Introduction}
For a smooth scheme $X$ over a perfect field $k$, we have natural isomorphisms,
\begin{equation}
\label{equation-chow-comparison}
H^{n,i}(X,\Z)\cong CH^i(X,2i-n)
\end{equation}
between the motivic cohomology groups \cite{MWV} on the left, and Bloch's higher Chow groups \cite{Bloch86} on the right. When $n=2i$, this relates motivic cohomology groups $H^{2i,i}(X,\Z)$ to the ordinary Chow groups $CH^i(X,0)=CH^i(X)$ of $X$. 
Such a comparison theorem for ordinary Chow groups was first proved by Voevodsky in \cite{FSV} under the assumption that $k$ admits resolution of singularities. It was extended to all higher Chow groups for $k$ perfect, again by Voevodsky, in \cite{MWV}. 
The motivic cohomology groups of $X$ are homomorphisms groups $Hom_{\DM{k,\Z}{}}(M(X),\Z(i)[n])$ from the motive $M(X)$ of $X$ to the motivic complexes $\Z(i)[n]$ in the triangulated category of motives, $\DM{k,\Z}{}$ (see \cite{MWV}).

The significance of the isomorphisms in (\ref{equation-chow-comparison}) is that we can translate various statements about Chow groups and intersection theory into statements about motivic cohomology. This is useful because statements and constructions involving (higher) Chow groups are often hard to prove (for example, Bloch's localisation sequence). However, arguments are greatly simplified if we consider analogous statements about motivic cohomology in $\DM{k,\Z}{}$. The two can then be related by Equation (\ref{equation-chow-comparison}).

Guided by this observation, in this article, we study the motives of algebraic stacks and derive various results about them. We also relate these motives to the Chow groups of algebraic stacks.

For algebraic stacks, the question of the correct notion of Chow groups is a subtle one. The first definition of Chow groups for algebraic stacks was given by Vistoli in \cite{VistoliIntersection} developing a rational intersection theory for stacks in the style of the Fulton-MacPherson theory for schemes. However, as there are not many invariant cycles on a stack, these Chow groups turn out to be quite small. Moreover, they lack expected properties like homotopy invariance, intersection product, etc. A more refined notion, as suggested by Totaro, was defined by Edidin and Graham in \cite{EG} for quotient stacks (and later generalised by Kresch in \cite{KreschChow} for all Artin stacks) which includes ``more cycles" and has better properties. These are called the Edidin-Graham-Totaro Chow groups and they are defined over integers. With $\Q$-coefficients, these two Chow groups agree. In this article, we relate the Edidin-Graham-Totaro Chow groups to the motivic cohomology of algebraic stacks (see Theorem \ref{chow-comparison}).

The story on the motivic side for algebraic stacks begins with Bertrand To\"{e}n. The notion of a motive for stacks was first defined by To\"{e}n for Deligne-Mumford stacks in \cite{Toen}. The general theory has been subsequently developed and extended by various authors (\!\!\cite{MotivesDM},\cite{HoyoisSixOperations},\cite{Joshua1,Joshua2}). In \cite{scholbach-shtukas}, a theory of motives for stacks is developed with a view of applications to number theory.
However, most of the existing literature on motives for algebraic stacks deals with $\Q$-coefficients and in the \'{e}tale topology. In \cite{Joshua2}, a comparison theorem similar to Theorem \ref{chow-comparison} is proved for the \textit{\'{e}tale} motivic cohomology and Totaro's Chow groups \textit{with $\Q$-coefficents}.  We use the notation $\DM{k,\Q}{\acute{e}t}$ for the corresponding triangulated category of \'{e}tale motives.

The Nisnevich topology is better suited for the study of Chow groups from the motivic perspective. This is evident from the comparison isomorphisms (\ref{equation-chow-comparison}) of Voevodsky. But as Nisnevich and \'{e}tale motivic cohomology agree over $\Q$, the choice of topology does not present any serious difficulty for comparing motivic cohomology with Chow groups rationally. However, all torsion information in the Chow groups is lost after tensoring with $\Q$. Hence, having a theory of motives in the Nisnevich topology with integral coefficients is desirable.

One definition for such a motive has already been given in \cite{hoskins2019} by Hoskins and Lehalleur. However, their definition employs the ``finite dimensional approximation" technique of Totaro in \cite{TotaroChow}, and is restricted to what they call \textit{exhaustive stacks}. This definition is applicable for quotient stacks as well as stacks like the moduli of vector bundles on smooth curves, but fails to hold in general.


The construction of a motive for an algebraic stack described in \cite{MotivesDM} (for the \'{e}tale topology) can also be carried out in the Nisnevich topology. This gives us a canonical definition of the motive for any algebraic stack locally of finite type over a field $k$. In this paper, we explore this construction. Note that while the construction defines the motive of an algebraic stack in great generality, it is ill-suited for the purpose of computations. Hence, we restrict ourselves to algebraic stacks that are ``Nisnevich locally quotient stacks", in the following sense:

\begin{definition}
	\label{definition-cd-quotient-stack}
	Let $\sX$ be an algebraic stack locally of finite type over a field $k$. We say that $\sX$ is a \textit{global quotient stack} if $\sX= [X/GL_n]$ for an algebraic space $X$. We say that $\sX$ is a \textit{cd-quotient stack} if it admits a Nisnevich covering $[X/GL_n]\rightarrow \sX$ by a global quotient stack\footnote{This nomenclature is inspired by \cite[Definition 2.1]{RydhApproximation}. Cd stands for \textit{completely decomposable}. Cd-topology is the old name for Nisnevich topology.}.
\end{definition}

For us, a Nisnevich covering of an algebraic stack will always mean a representable morphism (representable by \textit{algebraic spaces}) whose base change to any scheme is a Nisnevich covering. Note that we do not assume that the morphism is schematic (representable by schemes). Recall that a \emph{Nisnevich covering} $f: X\rightarrow Y$ of algebraic spaces is a surjective \'{e}tale morphism such that for any field-valued point $y:\Spec k\rightarrow Y$ there exists a lift $x: \Spec k\rightarrow X$ such that $y=f\circ x$.

The property of being a cd-quotient is enjoyed by a large class of algebraic stacks. Every global quotient stack is a cd-quotient stack. Further, for any quotient stack $[X/G]$ with $G$ a linear algebraic group, $X\times^G GL_n \rightarrow [X/G]$ is a $GL_n$-torsor, realising $[X/G]$ as a global quotient stack. 

Also, exhaustive stacks of \cite{hoskins2019} turn out to be cd-quotient stacks (see Section \ref{section-exhaustive}).

If $k$ is a perfect field then any stack with quasi-finite diagonal over $k$ is a cd-quotient stack (see \cite[\S 2]{conrad}). In particular, this includes all quasi-separated Deligne-Mumford stacks over $k$. 

Furthermore, if $\sX$ admits a good moduli space (in the sense of \cite{goodmodulispaces}), then by a theorem of Alper, Hall and Rydh, it is Nisnevich locally of the form $[\Spec(A)/GL_n]$ (see \cite[Theorem 13.1]{AHR19}), and hence a cd-quotient stack\footnote{In fact, this is generalised in \cite{ahhlr21} to any quasi-compact and quasi-separated stack with \textit{nice} stablizier groups.}.

The property of being a quotient stack Nisnevich locally gives us a greater handle on the stack from the computational point of view. In particular, we can construct a presentation of the stack in terms of certain simplicial schemes in the (unstable) $\A^1$-homotopy category $\sH_{\bullet}(k)$. More precisely, we have the following:

\begin{theorem}\label{motive-construction}
	Let $\sY\rightarrow \sX$ be a representable Nisnevich cover of algebraic stacks over a field $k$. Assume that $\sY$ is of the form $[Y/GL_n]$ for some algebraic space $Y$, and let $p:Y\rightarrow\sY\rightarrow\sX$ be the composite. Then, for the \v{C}ech nerve $Y_{\bullet}$ associated to $p $, $p_{\bullet}:Y_{\bullet}\rightarrow \sX$ is a Nisnevich local weak equivalence, i.e, $Y_{\bullet}\simeq \sX$ in $\sH_{\bullet}(k)$.
\end{theorem}

\begin{remark}\label{remark-local-epimorphism}
	As pointed to us by the referee, Theorem \ref{motive-construction} is true more generally: Let $p: Y\rightarrow \sX$ be a local epimorphim of simplicial presheaves. Denote by $N(p)$ the diagonal of the associated \v{C}ech bisimplicial object. Then the natural map $N(p)\rightarrow \sX$  is a local equivalence.\\
	In particular, Theorem \ref{motive-construction} continues to hold if the morphism $\sY\rightarrow \sX$ is taken to be a Nisnevich local epimorphism (as opposed to an honest Nisnevich covering) of simplicial presheaves.\\
	To maintain the geometric context we stick to $Y$ an algebraic space and $\sX$ an algebraic stack.
\end{remark}

The above theorem gives us a nice description of the motive of a cd-quotient stack in terms of the motive of the simplicial scheme $Y_{\bullet}$ in $\DM{k,\Z}{}$. This allows us to reduce many computations about motives to stacks to motives of (simplicial) schemes. In fact, we will use this to show that for quotient stacks, the cohomology groups of this motive agree with the Chow groups of Edidin-Graham-Totaro (see \cite{EG}, \cite{KreschChow}). 

\begin{theorem}
	\label{chow-comparison}
	Let $\sX:=[X/GL_r]$ be a quotient stack, where $X$ is a smooth quasi-projective scheme with a smooth action of $GL_r$. Then the Edidin-Graham-Totaro (higher) Chow groups and the motivic cohomology groups agree integrally, i.e, 
	\[CH^i(\sX, 2i-n)\simeq H^{n,i}(\sX,\Z).\]
Here $ H^{i,j}(\sX,\Z) := Hom_{\DM{k,\Z}{}}(M(\sX),\Z(j)[i])$. See \ref{definition-motive-stack} for the definition of $M(\sX)$.

\end{theorem}

The proof is quite straightforward and follows easily from \cite[Proposition 3.2]{krishna2012} and Theorem \ref{motive-construction}. The quasi-projective hypothesis is required to ensure that quotients of schemes by $GL_n$-actions remain schemes (see \cite[\S 2.2.1]{krishna2012}).



\begin{remark}\label{remark-smooth-nisnevich}
	In \cite{pirisi2018}, the notion of a smooth-Nisnevich covering is defined. A smooth-Nisnevich covering of a stack $\sX$ is a morphism $f:Y\rightarrow \sX$ with $Y$ a scheme such $f$ is smooth, surjective and any morphism $\Spec k\rightarrow \sX$ lifts to $Y$. By \cite{pirisi2018,aizenbud2019}, every quasi-separated algebraic stack admits a smooth-Nisnevich covering by a scheme. A smooth-Nisnevich covering is a local epimorphism in the Nisnevich topology since every Hensel local ring valued point also lifts. Hence, by Remark \ref{remark-local-epimorphism}, $Y_{\bullet}\rightarrow \sX$ is a Nisnevich local weak equivalence. In particular, a smooth-Nisnevich covering provides a presentation by simplicial schemes for any quasi-separated algebraic stack. Thus, the results in Sections \ref{section-motive-construction}, \ref{section-various-triangles} and \ref{section-applications-motivic-cohomology} also hold in the more general setting of (smooth) quasi-separated algebraic stacks.
\end{remark}

\noindent\textbf{Outline.} The paper is structured as follows: In Section \ref{section-motive-construction}, we define the Nisnevich motive for an algebraic stack following \cite{MotivesDM}. We then prove Theorem \ref{motive-construction}. The argument involves two key ideas: that every principal $GL_n$-bundle over a Henselian local ring is trivial, and that fibre products of stacks are a model for their homotopy fibre products in the model category of presheaves of groupoids (see \cite[Remark 2.3]{Holl}). Theorem \ref{motive-construction} gives us a computational handle on cd-quotient stacks which is used in the sequel to prove various results.

We then show, in Section \ref{section-chow-comparison}, that for quotients of quasi-projective schemes by $GL_n$, the integral motivic cohomology of an algebraic stack agrees with the Edidin-Graham-Totaro Chow groups (Theorem \ref{chow-comparison}). As stated earlier, such a result has already been proven in \cite[Theorem 3.5]{Joshua2} for the \'{e}tale motivic cohomology groups. The crucial difference being that in \cite{Joshua2} the comparison theorem is proved for rational Chow groups (and for the \'{e}tale topology) whereas we prove it for integral Chow groups (and in the Nisnevich topology). So, in a sense, the additional information we get from Theorem \ref{chow-comparison} is that the motivic cohomology groups and the Edidin-Graham-Totaro Chow groups also agree in their torsion parts.

In Section \ref{section-various-triangles}, we convince ourselves that various exact triangle for motives continue to hold for our construction. In particular, we establish Nisnevich descent (Proposition \ref{nisnevich-descent}), projective bundle formula (Theorem \ref{projective-bundle-formula}) and the Gysin triangle (Theorem \ref{gysin-triangle}) for cd-quotient stacks. This section adapts the corresponding results for the \'{e}tale topology in \cite{MotivesDM}.

One definition of motivic cohomology of schemes is as the Zariski hypercohomology of the motivic complexes $\Z(q)$ \cite[Proposition 14.16]{MWV}. In Section \ref{section-applications-motivic-cohomology}, we establish a similar result for cd-quotient stacks. Since stacks have very few Zariski open sets, the correct topology turns out to be the smooth-Nisnevich topology. We show that the smooth-Nisnevich hypercohomology of the motivic complexes $\Z(q)$ agrees with the homomorphism groups $H^{i,j}(\sX,\Z)$.

In Section \ref{section-exhaustive}, we compare our construction of the motive with the one in \cite{hoskins2019} for exhaustive stacks. We show that these two agree in $\DM{k,\Z}{}$. A similar comparison is proved for \'{e}tale motives in \cite[Appendix A]{hoskins2019}.

We include a technical result about model categories (Lemma \ref{homotopy-colimits-left-adjoints}) that is used for proving the projective bundle formula (Theorem \ref{projective-bundle-formula}) in an appendix. This is an elementary result which is surely known to anybody familiar with model categories (see also Remark \ref{remark-reference-homotopy-colimits}), and we only include it for our own edification.\\

\noindent\textbf{Acknowledgements.} The second-named author was supported by the INSPIRE fellowship (IF160348) of the Department of Science and Technology, Govt.\ of India during the course of this work. This work has also benefitted from the generous support of the Institut Mittag-Leffler (Swedish Research Council grant no. 2016-06596) while the second-named author was in residence there.\\
We would like to thank Roberto Pirisi for pointing out Remark \ref{remark-smooth-nisnevich}. We also thank the anonymous referee for his many helpful comments and suggestions.

\section{Nisnevich motive of an algebraic stack}
\label{section-motive-construction}
In this section we will define the motive of an algebraic stack as well as prove Theorem \ref{motive-construction}. In order to do this, the following observation will be crucial. 

\begin{remark}[Stacks as simplicial sheaves]
	\label{remark-stacks-as-simplicial-sheaves}
	Given an algebraic stack $\sX$ (or, more generally, a presheaf of groupoids), one associates a simplicial sheaf to $\sX$ as follows: $\sX$ defines a (strict) sheaf of groupoids $\sX \in Fun((Sch/k)^{op}, Grpds)$, which sends any $k$-scheme $U\mapsto \sX_U$ in the category of groupoids. Applying the nerve functor objectwise, we get a sheaf of simplicial sets. Let us briefly recall this procedure.\\
	Given a groupoid $\sX_U$ its nerve is a simplicial set $N(\sX_U)$ whose $k$-simplices are given by $k$-tuples of composable arrows,
	\[N(\sX_U)_k =\{A_0\overset{f_1}{\rightarrow} \ldots \overset{f_k}{\rightarrow} A_k \; | \; \text{$A_i$'s are objects and $f_i$ are morphisms in $\sX_U$}\}\]
	Note that $0$-simplices are just objects of $\sX_U$, while $1$-simplices are morphisms between them.\\
	The face maps $d_i: N(\sX_U)_k \rightarrow N(\sX_U)_{k-1}$ are given by composition of morphism at the $i$-th object (or deleting the $i$-th object for $i=0,k$). Similarly, the degeneracy maps $s_i: N(\sX_U)_k \rightarrow N(\sX_U)_{k+1}$ are given by inserting the identity morphism at the $i$-th object. Thus, we have a functor 
	\[Fun((Sch/k)^{op}, Grpds)\rightarrow \Delta^{op} PSh(Sm/k)\]
	from the category of presheaves of groupoids to the category of simplicial presheaves	(see \cite[Theorem 1.4]{HollPhD}). \\
	If $\sX\rightarrow \sZ$, $\sY\rightarrow \sZ$ are two morphisms of algebraic stacks, then viewing them as simplicial sheaves, one can form their homotopy fibre product $\sX \times^h_{\sZ} \sY$ in the homotopy category of simplicial presheaves. By \cite[Remark 2.3]{Holl}, the usual fibre product in the category of stacks $\sX \times_{\sZ} \sY$ serves as a model for this homotopy fibre product.
	
	In the remaining article, we will abuse notation by denoting the simplicial sheaf associated to a stack $\sX$ by $\sX$ itself.
	
\end{remark}

We will now recall the notion of the (unstable) $\A^1$-homotopy category over a field, as well as the construction of the motive of an algebraic stack as in \cite[Section 2]{MotivesDM}. 

Fix a base field $k$. Let $Sm/k$ denote the category of smooth schemes over $k$. Let $\Delta^{op}PSh(Sm/k)$ be the category of simplicial presheaves on $Sm/k$. Note that by Remark \ref{remark-stacks-as-simplicial-sheaves}, any stack $\sX$ can be considered as an object of $\Delta^{op}PSh(Sm/k)$.

$\Delta^{op}PSh(Sm/k)$ has a local model structure with respect to the Nisnevich topology (see \cite{Jar}). A morphism $f:X\rightarrow Y$ in $\Delta^{op}PSh(Sm/k)$ is a weak equivalence if the induced morphisms on stalks (for the Nisnevich topology) are weak equivalences of simplicial sets. Cofibrations are monomorphisms, and fibrations are characterised by the right lifting property. 

We Bousfield localise this model structure with respect to the class of maps $X\times \A^1\rightarrow X$ (see \cite[3.2]{MV}). The resulting model structure is called the Nisnevich motivic model structure. Denote by $\mathcal{H}_{\bullet}(k)$ the resulting homotopy category. This is the (unstable) $\A^1$-homotopy category for smooth schemes over $k$.


We will now define the Nisnevich motive of an algebraic stack. We begin by rapidly recalling the construction of the triangulated category of mixed motives.

Let $Cor_k$ denote the category of finite correspondences whose objects are smooth separated schemes over $k$. For any two $X,Y$, the morphisms of $Cor_k$ are given by $Cor(X,Y)$ which is the free abelian group generated by irreducible closed subschemes $W\subset X\times Y$ that are finite and surjective over $X$. An additive functor $F:Cor_k^{op}\rightarrow \mathbf{Ab}$ is called a presheaf with transfers. Let $PST(k,\Z)$ denote the category presheaves with transfers. For any smooth scheme $X$, let $\Z_{tr}(X)$ be the presheaf with tranfers which on any smooth scheme $Y$ is defined as
\[\Z_{tr}(X)(Y):=Cor(X,Y)\]
Let $K(PST(k,\Z))$ denote the category of complexes of presheaves with transfers. The category $K(PST(k,\Z))$ also has a Nisnevich motivic model structure which is defined analogously as in the case of $\Delta^{op}PSh(Sm/k)$. We denote the associated homotopy category by $\DM{k,\Z}{}$. This is Voevodsky's triangulated category of mixed motives in the Nisnevich topology (for details, see \cite{MWV}).  Then we have a functor
\[N\Z_{tr}(-): \Delta^{op}PSh(Sm/k)\rightarrow K(PST(k,\Z))\]
which sends a simplicial scheme $X_{\bullet}$ to its normalised chain complex $N\Z_{tr}(X_{\bullet})$. The $i$-th degree term of the chain complex $N\Z_{tr}(X_{\bullet})$ is given by $\Z_{tr}(X_i)$.

Since every simplicial presheaf is weakly equivalent to a simplicial scheme (see \cite{DHI}), this determines the derived functor of $N\Z_{tr}(-)$ completely. Note that $N\Z_{tr}(-)$ is a left Quillen functor and so it admits a left derived functor $M$ on the homotopy categories:
\[M:\mathcal{H}_{\bullet}(k)\rightarrow \DM{k,\Z}{}.\]
For a scheme $X$, $M(X)$ is the image of $\Z_{tr}(X)$ in $\DM{k,\Z}{}$.\\
The above adjunction is proved in \cite{MotivesDM} for the \'{e}tale model structure. The same proof works for the Nisnevich topology as well.

\begin{remark}
	\label{remark-motive-commutes-with-homotopy-colimits}
	$N\Z_{tr}$ is a left Quillen functor with respect to the Nisnevich motivic model structures on $\mathcal{H}_{\bullet}(k)$ and $\DM{k,\Z}{}$. Thus, Lemma \ref{homotopy-colimits-left-adjoints} implies that $M$ commutes with homotopy colimits.
\end{remark}

\begin{remark}
	The original contruction of $\DM{k,\Z}{}$ by Voevodsky uses the classical Grothendieck-Verdier derived category formalism. See \cite[Chapter 11]{cisinsi-deglise} for the construction via model categories.
\end{remark}

\begin{definition}[Motive of an algebraic stack]
	\label{definition-motive-stack}
	Let $\sX$ be an algebraic stack over $k$ thought of as a simplicial presheaf by the nerve construction outlined in Remark \ref{remark-stacks-as-simplicial-sheaves}. The motive of an algebraic stack $\sX$ is defined to be the image $M(\sX)$ in $\DM{k,\Z}{}$.\\
	Note that when $\sX$ is representable by a scheme $X$, $M(X)$ is the image of $\Z_{tr}(X)$ in $\DM{k,\Z}{}$.
\end{definition}

\begin{remark}
	The above construction was first done in \cite{MotivesDM} for the \'{e}tale model structure. However, it still goes through if we use the Nisnevich model structure instead. See also \cite{Joshua1} for an alternative approach.
\end{remark}

We will now show that cd-quotient stacks admit presentations by simplicial schemes in $\sH_{\bullet}(k)$. This is the content of Theorem \ref{motive-construction}. Having such a presentation will allow us to use homotopical descent techniques in the sequel in order to reduce various problems to the case of (simplicial) schemes. 

For the sake of clarity we will first prove Theorem \ref{motive-construction} in the case when $\sY = \sX$, i.e, when $\sX$ is a global quotient stack. Theorem \ref{motive-construction} is a minor extension of this case.

\begin{lemma}
	\label{lemma-principal-bundles-nisnevich-sections}
	Let $X$ be an algebraic space with an action of $GL_n$ and $\sX:=[X/GL_n]$ be the corresponding quotient stack. Let $X_{\bullet}$ denote the \v{C}ech nerve associated to $p:X\rightarrow\sX$. Then the map of simplicial presheaves $p_{\bullet}: X_{\bullet}\rightarrow \sX$ is a Nisnevich local weak equivalence.
\end{lemma}
\begin{proof}	
	It suffices to check that given a hensel local ring $\sO$, the induced map on $\sO$-points $p_{\bullet}: X_{\bullet}(\Spec\sO)\rightarrow \sX(\Spec\sO)$ is a weak equivalence of simplicial sets. Further, as a stack is a $1$-truncated simplicial set (being a groupoid valued functor), $\pi_i = 0$ for $i\geq 2$. Thus, we only need to verify that $p$ induces an isomorphism of homotopy groups for $i=0,1$.
	
	\noindent\underline{$i=0$:} Any map $\Spec\sO\rightarrow \sX$ gives rise to a $GL_n$-torsor $p':X\times_{\sX} \Spec\sO \rightarrow \Spec\sO$. As $\sO$ is a Hensel local ring, any $GL_n$-torsor over it is trivial and hence admits a section. This implies surjectivity on $\pi_0$.\\
	For injectivity, let $f_1, f_2:\Spec\sO\rightarrow X$ be two $\sO$-points of $X$ such that $p(f_1)=p(f_2)$, then there exists a map $F:\Spec\sO\rightarrow X\times_{\sX}X$ which after composing with each of the projection maps becomes $f_1$ and $f_2$, respectively. The map $F$ may be thought of as a 1-simplex in the simplicial set $X_{\bullet}(\Spec\sO)$, and therefore, corresponds to a map $\Delta^1\rightarrow X_{\bullet}(\Spec\sO)$, which gives a homotopy between the points $f_1$ and $f_2$ implying injectivity.
	
	\noindent\underline{$i=1$:} In this case, we need to show that for any $\Spec\sO$-valued point of $\sX$, the homotopy fibre product with $p_{\bullet}$ is contractible. Then, the long exact sequence of homotopy sheaves gives the required isomorphism. By \cite[Remark 2.3]{Holl}, the homotopy fibre product is precisely the fibre product in the category of stacks. This is a (\v{C}ech nerve of a) trivial $GL_n$-torsor over $\Spec\sO$, i.e, given a point $\Spec\sO\rightarrow\sX$, the homotopy fibre of $p_{\bullet}$ is precisely the \v{C}ech nerve $X_{\bullet}\times_{\sX}\Spec\sO$ corresponding to the $GL_n$-torsor $p':X\times_{\sX}\Spec\sO\rightarrow\Spec\sO$. Since $p'$ admits a section, the augmentation map $X_{\bullet}\times_{\sX}\Spec\sO\rightarrow\Spec\sO$ is a Nisnevich local weak equivalence. As, $\pi_i(\Spec\sO)=0$ for $i> 0$, we get the desired result. 
\end{proof}

Any Nisnevich cover $\sY\rightarrow\sX$ admits sections Nisnevich locally. Theorem \ref{motive-construction} now follows easily from this fact and Lemma \ref{lemma-principal-bundles-nisnevich-sections}.

\begin{proof}[Proof of Theorem \ref{motive-construction}]
	We need to check that $p_{\bullet}:Y_{\bullet}\rightarrow \sX$ induces an isomorphism on all homotopy sheaves, $\pi_i$. Further, it suffices to check this on all Hensel local schemes.
	
	\noindent\underline{$i=0$:} Let $\Spec\sO\rightarrow \sX$ be a point of $\sX$. Base changing, we get maps $Y\times_{\sX}\Spec\sO\rightarrow\sY\times_{\sX}\Spec\sO\rightarrow\Spec\sO$, where the first map is a principal $GL_n$-bundle and the second is a Nisnevich cover. So, the second map admits a Nisnevich local section. By the previous lemma, so does the first. This proves surjectivity.\\
	For injectivity, let $f_1,f_2:\Spec\sO\rightarrow Y$ be two points which map to the same point in $\sX$. This implies that there exists a section $\Spec\sO\rightarrow Y\times_{\sX}Y$ which after composing with each of the projection maps becomes $f_1$ and $f_2$, respectively.
	
	\noindent\underline{$i>0$:} To show this we need to show that for any $\Spec\sO$-valued point of $\sX$, the homotopy fibre product with $p_{\bullet}$ is contractible. Then, the long exact sequence of homotopy sheaves gives us the required isomorphism.\\
	As noted in the previous lemma, the homotopy fibre product is equal to the stacky fibre product\cite[Remark 2.3]{Holl}. Hence, for a point $\Spec\sO\rightarrow\sX$, the homotopy fibre of $p_{\bullet}$ is precisely the \v{C}ech nerve $Y_{\bullet}\times_{\sX}\Spec\sO$ of $\Spec\sO$. Since it admits a section, the augmentation map $Y_{\bullet}\times_{\sX}\Spec\sO\rightarrow\Spec\sO$ is a Nisnevich local weak equivalence. As, $\pi_i(\Spec\sO)=0$ for $i> 0$, we get the desired result.
\end{proof}

\begin{remark}
	\label{remark-gln-presentation}
	For a cd-quotient stack, we have a \v{C}ech nerve $Y_{\bullet}\rightarrow \sX$ which is a Nisnevich local weak equivalence, by Theorem \ref{motive-construction}. Applying the functor $M:\mathcal{H}_{\bullet}(k)\rightarrow \DM{k,\Z}{}$, we see that the motive of a cd-quotient stack $\sX$ is given by the normalised chain complex $N\Z_{tr}(Y_{\bullet})$.
\end{remark}

\begin{definition}
	For a cd-quotient stack $\sX$, let $p:X\rightarrow\sX$ be the presentation obtained from the composition $X\rightarrow [X/GL_n]\rightarrow\sX$ (as in the hypothesis of Theorem \ref{motive-construction}), and let $p_{\bullet}: X_{\bullet}\rightarrow \sX$ denote the associated \v{C}ech nerve. Motivated by the content of Theorem \ref{motive-construction}, we will call $p_{\bullet}: X_{\bullet}\rightarrow \sX$ a \textit{$GL_n$-presentation} of $\sX$. 
\end{definition}

\begin{remark}\label{motive-not-geometric}
	The category generated by motives of stacks as constructed above is larger than the category of geometric motives (motives generated by smooth quasi-projective schemes). In fact, if $G$ is a finite group, $BG$ is not a geometric motive. To see this note that any realization of a geometric motive must have bounded cohomology. Further, for a finite group, the cohomology of $BG$ is the same as the group cohomology of the group $G$. But if $G$ is finite cyclic, then the latter is periodic in odd degrees, showing that $BG$ does not have bounded cohomology for a cyclic group.
\end{remark}

\begin{proposition}
	\label{proposition-comparison-etale-motive}
	Let $\sX$ be a cd-quotient stack. The Nisnevich motive $M(Y_{\bullet})$ of $\sX$ agrees with its \'{e}tale motive in $\DM{k,\Z}{\acute{e}t}$ after \'{e}tale sheafification.
\end{proposition}
\begin{proof}
	Let $Y_{\bullet}\rightarrow \sX$ be a $GL_n$-presentation so that $M(Y_{\bullet})\simeq M(\sX)$ in $\DM{k,\Z}{}$. Since sheafification is an exact functor, we have a functor $\sigma: \DM{k,\Z}{}\rightarrow \DM{k,\Z}{\acute{e}t}$ which takes Nisnevich local weak equivalences to \'{e}tale local weak equivalences (see \cite[Remark 14.3]{MWV}). Thus, the Nisnevich local weak equivalence $M(Y_{\bullet})\rightarrow M(\sX)$ becomes an \'{e}tale local weak equivalence after \'{e}tale sheafification. Hence, $M(Y_{\bullet})\simeq M(\sX)$ in $\DM{k,\Z}{\acute{e}t}$.
	
	In fact, more is true. We can show that any $GL_n$-presentation is equivalent to the \v{C}ech nerves associated to a smooth presentation of $\sX$. 
	
	Let $p: U\rightarrow \sX$ be a smooth presentation, and $p_{\bullet}: U_{\bullet}\rightarrow \sX$ be the associated \v{C}ech nerve. Given any strict Hensel local point $\Spec \sO\rightarrow \sX$ consider the base change $p_\sO : U_{\sO}\rightarrow \Spec \sO$. This is a smooth morphism and so admits \'{e}tale local sections. In fact, since $\sO$ is a Hensel local ring, we have a section $\Spec \sO \rightarrow U_{\sO}$ of $p_\sO$. This implies that the induced map of simplicial sets $p_{\bullet}(\sO): U_{\bullet}(\sO)\rightarrow \sX(\sO)$ is a weak equivalence of simplicial sets (proof is similar to Lemma \ref{lemma-principal-bundles-nisnevich-sections}). Thus, $U_{\bullet}\rightarrow \sX$ is an \'{e}tale local weak equivalence. So the induced map on \'{e}tale motives $M(U_{\bullet}) \rightarrow M(\sX)$ is also an \'{e}tale local weak equivalence in $\DM{k, \Z}{\acute{e}t}$ (see \cite[Corollary 2.14]{MotivesDM}). Hence, $M(Y_{\bullet})\simeq M(U_{\bullet})$ in $\DM{k,\Z}{\acute{e}t}$.
\end{proof}

\begin{remark}
In Proposition \ref{proposition-comparison-etale-motive}, the $GL_n$-presentation $Y_{\bullet}$ and the \v{C}ech nerve $U_{\bullet}$ are simplicial objects in the category of \textit{algebraic spaces}. This is because a smooth presentation $p:U\rightarrow \sX$ of an algebraic stack need not be representable by schemes, but only algebraic spaces. However, as any algebraic space admits a Nisnevich presentation by a scheme \cite[Theorem II.6.4]{knutson}, we can refine the \v{C}ech nerve $U_{\bullet}$ to a generalised hypercovering $V_{\bullet}$ such that each $V_i$ is a scheme. Then $M(V_{\bullet})$ computes the motive $M(\sX)$ (see \cite{DHI} for details).
\end{remark}

\begin{remark}[\'{E}tale motives with $\Q$-coefficients]
	Tensoring with $\Q$ gives us a functor $-\otimes \Q: \DM{k,\Z}{\acute{e}t}\rightarrow \DM{k,\Q}{\acute{e}t}$ which is just change of coefficients (this also works in the Nisnevich topology). We write $M(Y_{\bullet})\otimes \Q := M(Y_{\bullet})_{\Q}$. By \cite[Theorem 14.30]{MWV}, the \'{e}tale sheafification functor $\sigma: \DM{k,\Q}{}\rightarrow \DM{k,\Q}{\acute{e}t}$ is an equivalence of categories.\\
	By \cite[Theorem 4.6]{MotivesDM}, for a smooth separated Deligne-Mumford stack $\sX$, $M(\sX)_{\Q}$ is a geometric motive. This does not contradict Remark \ref{motive-not-geometric}, since the cohomology groups of a cyclic group are torsion and will vanish after tensoring with $\Q$.\\
	Further, if $\pi:\sX\rightarrow X$ is the coarse space map, then $M(\pi)_\Q: M(\sX)_\Q\rightarrow M(X)_\Q$ is an isomorphism by \cite[Theorem 3.3]{MotivesDM}. This is clearly false integrally, since for a finite group $G$ over a field $k$, the structure map $BG\rightarrow \Spec k$ is a coarse space map.
\end{remark}

\section{Comparison with Edidin-Graham-Totaro Chow Groups}
\label{section-chow-comparison}

We will now proceed to show that the motivic cohomology groups of the motive defined by Theorem \ref{motive-construction} agree with the (higher) Chow groups defined by Edidin-Graham-Totaro for quotients of smooth quasi-projective schemes in \cite{EG}. This implicitly follows from \cite{krishna2012}, but we write out the details for the sake of completeness.
In what follows we only consider the action of $G:=GL_r$ on \textit{quasi-projective schemes}.

We begin by recalling Totaro's definition of Chow groups for quotient stacks (see \cite[Section 2.2]{EG}). The definition is via a ``Borel type construction". 

Let $X$ be an $n$-dimensional smooth quasi-projective scheme with an action of $GL_r$ and let $[X/GL_r]$ be quotient of this action. Choose an $l$-dimensional representation $V$ of $GL_r$ such that $V$ has an open subset $U$ on which $GL_r$ acts freely and whose complement has codimension greater than $n-i$. Then, we have a principal $GL_r$-bundle $X\times U\rightarrow (X\times U)/GL_r$, and we denote the quotient as $X_{GL_r}:= (X\times U)/GL_r$. Note that $X_{GL_r}$ is a quasi-projective scheme (see \cite[Lemma 2.2]{krishna2012} or \cite[Proposition 23]{EG}). 
\begin{definition}
\label{totaro-chow}
With set-up as above, we define the $i$-th Chow group of $[X/GL_r]$ as 
\[CH^i([X/GL_r]) :=CH^{i}(X_{GL_r}).\]
This definition is independent of the choice of $V$ and $U$ so long as $\codim(V,U)> i$ (see \cite[section 2.2]{EG} for details).
\end{definition}
Using Bloch's cycle complex, we can extend this definition to higher Chow groups in a similar manner.

The following definition is a special case of the one in \cite[Section 4.2]{MV}:

\begin{definition}{\cite[Definition 2.1]{krishna2012}}
	A pair $(V,U)$ of smooth schemes over $k$ is said to be a \textit{good pair} for $G$ if $V$ is a  $k$-rational representation of $G$ and $U\subset V$ is a $G$-invariant open subset on which $G$ acts freely and the quotient $U/G$ is a smooth quasi-projective scheme.\\
	A sequence of pairs $\rho=(V_i,U_i)_{i\geq 1}$ is said to be an \textit{admissible gadget} for $G$ if there exists a good pair $(V,U)$ for $G$ such that $V_i=V^{\oplus i}$ and $U_i\subset V_i$ is a $G$-invariant open subscheme such that the following hold:
	\begin{itemize}
		\item $(U_i\oplus V)\cup (V\oplus U_i)\subseteq U_{i+1}$ as $G$-invariant open subsets.
		\item $codim_{U_{i+2}}(U_{i+2}\setminus (U_{i+1}\oplus V))> codim_{U_{i+1}}(U_{i+1}\setminus (U_{i}\oplus V))$.
		\item $codim_{V_{i+1}}(V_{i+1}\setminus U_{i+1})> codim_{V_i}(V_i\setminus U_i)$.
		\item The action of $G$ on $U_i$ is free, and the quotient is quasi-projective.
	\end{itemize}
\end{definition}

An example of such an admissible gadget can be given as follows. Let $V$ be a $k$-rational representation of $GL_n$, and let $U$ be a $GL_r$-invariant open subset on which $GL_r$ acts freely, and the quotient $U/GL_r$ is a quasi-projective scheme. Then $(V,U)$ is a good pair, and we define an admissible gadget $\rho=(V_i,U_i)_{i\geq 1}$ by taking $U_{i+1}:=(U_i\oplus V)\cup (V\oplus U_i)$.\\
For a quasi-projective scheme $X$ with an action $G$, consider the mixed quotients $X^i(\rho):= X\overset{G}{\times} U_i$, and define $X_G(\rho):=\colim X\overset{G}{\times} U_i$. While $
 X\overset{G}{\times} U_i$ is an ordinary colimit, since it is taken over a filtered category, it is also a homotopy colimit (see \cite[Example 12.3.5]{BK}). Further, we denote by $X^{\bullet}_G$ the simplicial scheme:
\begin{center}
\begin{tikzcd}
  \ldots\arrow[r]\arrow[r, shift left]\arrow[r, shift right]\arrow[r, shift right=1.5ex] & G\times G\times X \arrow[r]\arrow[r,shift left]\arrow[r, shift right] & G\times X\arrow[r,shift left]\arrow[r] & X\\	
\end{tikzcd}
\end{center}
The following lemma relates $X_G(\rho)$ with $X^{\bullet}_G$.

\begin{lemma}{\cite[Proposition 3.2]{krishna2012}}
Let $\rho =(V_i,U_i)_{i\geq 1}$ be an admissible gadget for a linear algebraic group $G$ over $k$. For any quasi-projective $G$-scheme $X$, there is a canonical isomorphism $X_G(\rho)\cong X^{\bullet}_G$ in $\mathcal{H}_{\bullet}(k)$.
\end{lemma}

The above lemma gives us a comparison theorem between (higher) Chow groups and motivic cohomology (see also \cite[Theorem 3.5]{Joshua2} for a version in the \'{e}tale topology with $\Q$-coefficients).


\begin{proof}[Proof of Theorem \ref{chow-comparison}]
	Note that $GL_r\times X$ is isomorphic to $X\times_{\sX}X$. Thus, the simplicial scheme $X^{\bullet}_{GL_r}$ is isomorphic to the \v{C}ech nerve $X_{\bullet}\rightarrow\sX$. By Theorem \ref{motive-construction}, $X_{\bullet}\rightarrow\sX$ is a Nisnevich local equivalence.
	
	Let $\rho=(V_i,U_i)$ be an admissible gadget for $GL_r$, with $\dim (V)=l$, $m:=r^2$. By definition, 
	\[CH^i(\sX, 2i-n)=CH^{i}(X^N(\rho),2i-n) \simeq Hom_{\DM{k,\Z}{}}(M(X^N(\rho)),\Z(i)[n])\] 
	for $N$ sufficiently large. Here, the second equivalence follows from Equation \ref{equation-chow-comparison} and the fact that $H^{n,i}(X,\Z)=Hom_{\DM{k,\Z}{}}(M(X), \Z(i)[n])$. Note that this is well-defined since the maps $X^s(\rho)\rightarrow X^t(\rho)$ are vector bundles and so have the same Chow groups by $\A^1$-invariance. 
	
	Now, as $X_{GL_r}(\rho)=\underset{N}{\colim} X^N(\rho)$ is a filtered colimit, by \cite[Example 12.3.5]{BK} it is also a homotopy colimit. Then, by Remark \ref{remark-motive-commutes-with-homotopy-colimits},
	\begin{align*}
	Hom_{\DM{k,\Z}{}}(M(X_{GL_r}(\rho)), \Z(i)[n])&= 
	Hom_{\DM{k,\Z}{}}(M\big(\underset{N}{\hocolim} X^N(\rho)\big),\Z(i)[n])\\
	&\simeq Hom_{\DM{k,\Z}{}}\big(\underset{N}{\hocolim}M(X^N(\rho)),\Z(i)[n]\big)\\
	&\simeq \underset{N}{\holim}Hom_{\DM{k,\Z}{}}\big(M(X^N(\rho)),\Z(i)[n]\big).
	\end{align*}
	Since the maps $X^n(\rho)\rightarrow X^m(\rho)$ are $\A^1$-invariant, the groups $Hom_{\DM{k,\Z}{}}\big(M(X^N(\rho)),\Z(i)[n]\big)$ stabilise for all $m\geq N$. Thus, 
	\[Hom_{\DM{k,\Z}{}}(M(X_G(\rho)), \Z(i)[n])\simeq Hom_{\DM{k,\Z}{}}\big(M(X^N(\rho)),\Z(i)[n]\big)\]
	whenever $N$ is large enough. Further, we also have the relations,
	\[H^{n,i}(\sX, \Z)=Hom_{\DM{k,\Z}{}}(M(\sX),\Z(i)[n])\simeq Hom_{\DM{k,\Z}{}}(M(X_{\bullet}), \Z(i)[n])\]
	in $\DM{k,\Z}{}$. By the previous lemma, $X_G(\rho)$ and $X_{\bullet}$ are isomorphic in $\mathcal{H}_{\bullet}(k)$. Putting these together, we get required isomorphism.
\end{proof}

\begin{remark}
	Equivariant algebraic cobordisms (see \cite{heller13}, \cite{krishna2012}) are defined by a Borel type construction analogous to definition of equivariant (higher) Chow groups.
	By the above considerations, one can think of equivariant algebraic cobordism as an algebraic cobordism of the associated quotient stack.
\end{remark}

\begin{remark}
 In fact, Theorem \ref{chow-comparison} also holds when $X$ is a smooth scheme or more generally an algebraic space that is the quotient of a smooth scheme by $GL_n$. The proof uses a modified version of \cite[Proposition 3.2]{krishna2012} which holds for all algebraic spaces (See \cite[\S 6.1.2, 6.1.3]{deshmukh2021}.
\end{remark}

\section{Various Triangles}
\label{section-various-triangles}

We now establish Nisnevich descent and blow-up sequence for the Nisnevich motive and also prove the projective bundle formula. As a consequence of the projective bundle formula, we get a Gysin triangle for cd-quotient stacks. These result are already known for \'{e}tale motives (see \cite{MotivesDM}). All the arguments in this section are directly adapted from their \'{e}tale counterparts in \cite{MotivesDM} $-$ except for the projective bundle formula. The argument for the projective bundle formula in the \'{e}tale case relies on the identification of the Picard group with $H_{\acute{e}t}^2(\sX,\Z(1))$. This identification fails for stacks if \'{e}tale topology is replaced by Nisnevich topology. So we adopt a different approach using homotopical descent.

In this section, we work exclusively with cd-quotient stacks (however, see Remark \ref{remark-smooth-nisnevich}).

\begin{remark}\label{BaseChange}
Let $\sZ$ be a cd-quotient stack. If $\sY\rightarrow\sZ$ is a representable morhpism, then $\sY$ is also a cd-quotient stack. To see this, let $Z\rightarrow [Z/GL_n]\rightarrow\sZ$ be a Nisnevich covering of $\sZ$ by a quotient stack. Now, observe that we have a cartesian diagram by base change,
\begin{center}
\begin{tikzcd}
	\sY\times_{\sZ}Z\arrow[r]\arrow[d] & Z\arrow[d]\\
	\sY\times_{\sZ}[Z/GL_n]\arrow[r]\arrow[d] & \left[ Z/GL_n\right]\arrow[d]\\
	\sY\arrow[r] & \sZ
\end{tikzcd}
\end{center}
where $\sY\times_{\sZ}Z\rightarrow \sY\times_{\sZ} [Z/GL_n]$ is a $GL_n$-torsor and $\sY\times_{\sZ} [Z/GL_n]\rightarrow\sY$ is a Nisnevich cover. Denote the algebraic space $\sY\times_{\sZ}Z$ by $Y$. Thus, $\sY$ is a cd-quotient stack. Hence, by Theorem \ref{motive-construction}, $Y_{\bullet}\rightarrow \sY$ is a Nisnevich local weak equivalence.\\
This tells us that \textit{$GL_n$-presentations} respect base change.
\end{remark}

\begin{proposition}\label{nisnevich-descent} For a distinguished Nisnevich square,
\begin{center}
\begin{tikzcd}
	\sW\arrow[r]\arrow[d] & \sY\arrow[d,"p"]\\
	\sX\arrow[r,hook,"j"]	&\sZ
\end{tikzcd}
\end{center}
where $j$ is an open immersion and $p$ is \'{e}tale representable, the induced diagram on motives is homotopy cartesian.
\end{proposition}

\begin{proof}
For a distinguished Nisnevich square of stacks
\begin{center}
\begin{tikzcd}
	\sW\arrow[r]\arrow[d] & \sY\arrow[d,"p"]\\
	\sX\arrow[r,hook,"j"]	&\sZ
\end{tikzcd}
\end{center}
we need to show that the following induced diagram of motives is homotopy catesian 
\begin{center}
	\begin{tikzcd}
		M(\sW)\arrow[r,]\arrow[d,] & M(\sY)\arrow[d,"p"]\\
		M(\sX)\arrow[r,"j"]	&M(\sZ).
	\end{tikzcd}
\end{center}
For this we argue as follows:
if $Z\rightarrow [Z/GL_n]\rightarrow\sZ$ is a Nisnevich covering of $\sZ$ by a $GL_n$-torsor, then by Remark \ref{BaseChange}, we can base change this covering to $\sY,\sX$ and $\sW$. This gives us a cartesian diagram of hypercovers:
\begin{center}
\begin{tikzcd}
	W_{\bullet}\arrow[r]\arrow[d] & Y_{\bullet}\arrow[d,"p_{\bullet}"]\\
	X_{\bullet}\arrow[r,"j_{\bullet}"]	&Z_{\bullet}
\end{tikzcd}
\end{center}
Thus, for each $i$, we have a cartesian diagram,
\begin{center}
	\begin{tikzcd}
		W_{i}\arrow[r]\arrow[d] & Y_{i}\arrow[d,"p_{i}"]\\
		X_{i}\arrow[r,"j_{i}"]	&Z_{i}.
	\end{tikzcd}
\end{center}
Note that the map $j_i: X_i:=\sX\times_{\sZ} Z_i\rightarrow Z_i$ is a base change of the map $j:\sX\rightarrow \sZ$, and hence is an open immersion. By similar reasoning, the map $p_i: Y_i:=\sY\times_{\sZ} Z_i\rightarrow Z_i$ is \'{e}tale. Moreover, we have an isomorphisms $(Z_i\setminus X_i)\simeq (\sZ\setminus \sX)\times_{\sZ} Z_i$ and $p^{-1}(\sZ\setminus \sX)\times_{\sZ}Z_i\simeq p_i^{-1}(Z_i\setminus X_i)$, as reduced closed subschemes. By hypothesis, we have a Nisnevich distinguished square of stacks. Hence, $p^{-1}(\sZ\setminus \sX)\simeq \sZ\setminus \sX$ which implies that $p_i^{-1}(Z_i\setminus X_i)\simeq Z_i\setminus X_i$. This shows that, for each $i$, the above the diagram of smooth schemes is a distinguished Nisnevich square.\\
Thus, for each $i$, the following diagram of motives is homotopy (co)cartesian,
\begin{center}
\begin{tikzcd}
	M(W_i)\arrow[r]\arrow[d] & M(Y_i)\arrow[d]\\
	M(X_i)\arrow[r]	&M(Z_i)
\end{tikzcd}
\end{center}
By Remark \ref{remark-motive-commutes-with-homotopy-colimits}, $M(Z_{\bullet})\simeq \hocolim M(Z_i)$. As homotopy colimits commute with homotopy colimits, the following diagram is again homotopy (co)cartesian,
\begin{center}
\begin{tikzcd}
	M(W_{\bullet})\arrow[r]\arrow[d] & M(Y_{\bullet})\arrow[d]\\
	M(X_{\bullet})\arrow[r]	&M(Z_{\bullet})
\end{tikzcd}
\end{center}
By Theorem \ref{motive-construction}, we get the required result.
\end{proof}

\begin{theorem}[Projective Bundle Formula]
\label{projective-bundle-formula}
Let $\sE$ be a vector bundle of rank $n+1$ on a stack $\sX$. There exists a canonical isomorphism in $\DM{k,\Z}{}$:
\[M(\Proj(\sE))\rightarrow\bigoplus_{i=0}^n M(\sX)(i)[2i]\]
\end{theorem}
\begin{proof}
	As projective bundle formula is known for smooth schemes by \cite[Theorem 15.12]{MWV}, we will deduce the result for stacks by a homotopical descent argument. To make such a homotopical descent argument, we need to ensure that homotopy colimits commute with $M$ and derived tensor. But both $M$ and derived tensor are derived functors of left Quillen functors, so Lemma \ref{homotopy-colimits-left-adjoints} ensures that this is true.
	 
	Let $p:\Proj(\sE)\rightarrow \sX$ be the projective bundle, and $\sO(1)$ the canonical line bundle on it. This construction behaves well with respect to base change. If $U_{\bullet}\rightarrow \sX$ is a $GL_n$-presentation, then by base change we get projective bundles $p_i:V_i \rightarrow U_i$ for every $i$, and line bundles $\sO(1)_{V_i}$ on $V_i$ by pullback. Moreover, by Remark \ref{BaseChange}, the \v{C}ech nerve $V_{\bullet}\rightarrow \Proj(\sE)$ is a $GL_n$-presentation of $\Proj(\sE)$. As each $p_i: V_i\rightarrow U_i$ is a projective bundle, by the projective bundle formula (see \cite[Theorem 15.12]{MWV}), we have:
	\begin{equation}
	M(V_i) \simeq \oplus_{j=0}^n M(U_i) \otimes \Z(j)[2j],
	\end{equation}
	in $\DM{k,\Z}{}$.
	
	Since $M$ and derived tensor are left Quillen, using Lemma \ref{homotopy-colimits-left-adjoints}, we have
	\begin{align*}
	M(\Proj(\sE)) \simeq  M( \hocolim V_i )	&\simeq \hocolim (M( V_i ))\\
	&\simeq \hocolim (\oplus_{j=0}^n M( U_i ) \otimes \Z(j)[2j])\\
	&\simeq \oplus_{j=0}^n (M( \hocolim U_i ) \otimes \Z(j)[2j])\\
	&\simeq \oplus_{j=0}^n M(\sX) \otimes \Z(j)[2j],
	\end{align*}
	as required.
\end{proof}

\begin{remark}
In fact, the above theorem works for any simplicial presheaf $\sX$. This was pointed out to us by a referee. By \cite[\S 2]{dug}, any simplicial presheaf is quasi-isomorphic to a simplicial presheaf which is levelwise a disjoint union of representables. That is, we have a quasi-isomorphism $X_{\bullet}\rightarrow \sX$, such that each $X_i$ is a disjoint union of schemes. Now the proof above works verbatim.\\
The only missing piece is to define an appropriate notion of a vector bundle and projective bundle for simplicial presheaves. For this, one can proceed as follows: for any simplicial presheaf $\sX$, a principal $GL_n$-bundle is given by a morphism $\sX\rightarrow BGL_n$. In fact, the homotopy fibre product $\sX\times_{BGL_n}^h \Spec k:=\sY$ defines a the total space of the principal $GL_n$-bundle over $\sX$. Let $V$ be an $n$-dimensional vector space over $k$. The associated fibre space $\sY\times^{GL_n} V:= \sE$ gives a vector bundle on $\sX$. 

Similarly, if we take the fibre space associated to $\P(V)$, we get a projective bundle $\Proj(\sE)\rightarrow \sX$ for any simplicial presheaf. These constructions are compatible with base change and agree with the usual notions of vector bundles and projective bundles when $\sX$ is an algebraic stack. 

\end{remark}

\begin{proposition}
	Let $\sZ\subset \sX$ be a smooth closed substack of $\sX$. Let $Bl_{\sZ}(\sX)$ denote the blow-up of $\sX$ in the centre $\sZ$, and $\sE$ be the exceptional divisor. Then we have a canonical distinguised triangle:
	\[M(\sE)\rightarrow M(\sZ)\oplus M(Bl_{\sZ}(\sX))\rightarrow M(\sX)\rightarrow M(\sE)[1]\]
\end{proposition}
\begin{proof}
	Let $X\rightarrow [X/GL_n]\rightarrow\sX$ be a Nisnevich covering of $\sX$ by a $GL_n$-torsor. Since the morphism $Bl_{\sZ}(\sX)\rightarrow\sX$ is projective, it is representable. Then, we can base change $X$ to $Bl_{\sZ}(\sX),\sZ$ and $\sE$. The rest of the proof is the same as Proposition \ref{nisnevich-descent}.
\end{proof}

\begin{theorem} \label{blow-up-formula}
	Let $\sX$ be a smooth stack and $\sZ \subset \sX$ be a smooth closed substack of pure codimension $c$. Then,
	\[M(Bl_{\sZ}(\sX))\simeq M(\sX)\oplus_{i=0}^{c-1} M(\sZ)(i)[2i]\]
\end{theorem}
\begin{proof}
	Using the previous result, we have a canonical distinguished triangle:
	\[M(\sE)\rightarrow M(\sZ)\oplus M(Bl_{\sZ}(\sX))\rightarrow M(\sX)\rightarrow M(\sE)[1],\]
	where $p:Bl_{\sZ}(\sX)\rightarrow\sX$ is the blow-up. The exceptional divisor is the projectivisation of the normal bundle $N_{\sZ}(\sX)$ of $Z$ in $\sX$, i.e, $\sE\simeq \Proj(N_{\sZ}(\sX))$. If $M(\sX)\rightarrow M(\sE)[1]$ is zero, then the projective bundle formula for $\Proj(N_{\sZ}(\sX))$ gives us the result. 
	
	To prove that $M(\sX)\rightarrow M(\sE)[1]$ is zero, we argue exactly as in \cite[Theorem 3.7]{MotivesDM} (see also \cite[Chapter 5, Proposition 3.5.3]{FSV}). Take $\sX\times \A^1$ and consider the blow-up along $\sZ\times \lbrace 0\rbrace$. We have a map $q:Bl_{\sZ\times\lbrace 0\rbrace}(\sX\times \A^1)\rightarrow \sX\times \A^1$. Consider the morphism of exact triangles,
	\begin{center}
	\begin{tikzcd}
		M(\sE)\arrow[r]\arrow[d] & M(q^{-1}(\sZ\times\lbrace 0\rbrace))\arrow[d]\\
		M(\sZ)\oplus M(Bl_{\sZ}(\sX))\arrow[r]\arrow[d] & M(\sZ\times \lbrace 0\rbrace)\oplus M(Bl_{\sZ\times \lbrace 0\rbrace}(\sX\times \A^1))\arrow[d,"f"]\\
		M(\sX)\arrow[r,"s_0"]\arrow[d,"g"] & M(\sX\times \A^1)\arrow[d,"h"]\\
		M(\sE)[1]\arrow[r,"a"] & M(q^{-1}(\sZ\times\lbrace 0\rbrace))[1]\\
	\end{tikzcd}
	\end{center}
	By the projective bundle formula, the morphism $a$ is split injective, and $s_0$ is an isomorphism. Hence, to show that $g$ is zero, it suffices to show that $h$ is zero. To see this, note that the composition
	\[M(\sX\times \lbrace 1\rbrace)\rightarrow M(Bl_{\sZ\times \lbrace 0\rbrace}(\sX\times \A^1))\rightarrow M(\sX\times \A^1)\]
	is an isomorphism. This implies that $f$ admits a section so $h$ must be zero.
\end{proof}

Given a smooth separated stack over a field, it is a difficult problem to determine whether it is a quotient stack, or equivalently, whether it has the resolution property \cite{Tot, GrossRes}. The following corollary shows that while the resolution property may be hard to establish, the motive of every cd-quotient stack is a direct summand of the motive of a quotient stack. In a sense, this says that the intersection theory of cd-quotient stacks can be ``captured" by quotient stacks to some extent.

\begin{corollary}
	Let $\sX$ be a smooth separated Deligne-Mumford  stack over a field of characteristic zero. Then the motive of $\sX$ is the direct summand of the motive of a quotient stack.
\end{corollary}

\begin{proof}
	 The motive of a smooth separated Deligne-Mumford stack over a field of characteristic zero is a direct summand of a smooth separated Deligne-Mumford stack with quasi-projective coarse moduli space by \cite[Theorem 4.3]{MotivesDM} and Theorem \ref{blow-up-formula}. Now, \cite[Theorem 4.4]{KreschGeometry}
 completes the proof.
\end{proof}

\begin{definition}
For a map $M(X)\rightarrow M(Y)$ of motives of stacks (or simplicial schemes), we denote the cone by
\[M\Big(\frac{X}{Y}\Big):= cone(M(X)\rightarrow M(Y))\;\; \text{in} \;\DM{k,\Z}{}.\]
\end{definition}

\begin{lemma}
	Let $f:\sX'\rightarrow\sX$ be an \'{e}tale representable morphism of algebraic stacks, and let $\sZ\subset \sX$ be a closed substack such that $f$ induces an isomorphism $f^{-1}(\sZ)\simeq \sZ$. Then the canonical morphism
	\[M\Big(\frac{\sX'}{\sX'\setminus\sZ}\Big)\rightarrow M\Big(\frac{\sX}{\sX\setminus\sZ}\Big)\]
	is an isomorphism.
\end{lemma}
\begin{proof}
	Let $U_{\bullet}\rightarrow\sX$ be a $GL_n$-presentation. Let $V_{\bullet}:=U_{\bullet}\times_{\sX} \sX'$ be a $GL_n$-presentation of $\sX'$ obtained by base change. We have a map of simplicial sets $f_{\bullet}:V_{\bullet}\rightarrow U_{\bullet}$ induced by the map $f:\sX' \rightarrow \sX$. Note that for every simplicial degree, $f_i$ is \'{e}tale and that $f_i^{-1}(\sZ\times_{\sX} U_i)\simeq \sZ\times_{\sX} U_i$. Let $Z_i$ denote the base change $\sZ\times_{\sX}U_i$. Then, by \cite[Chapter 3, Proposition 5.18]{FSV}, the canonical morphism
	\[M\Big(\frac{V_{\bullet}}{V_{\bullet}\setminus Z_{\bullet}}\Big)\rightarrow M\Big(\frac{U_{\bullet}}{U_{\bullet}\setminus Z_{\bullet}}\Big).\]
	is an isomorphism. Now, Theorem \ref{motive-construction} gives us the required result.
\end{proof}

\begin{corollary}
	Let $p:V\rightarrow\sX$ be a vector bundle of rank $n$ over an algebraic stack. Denote by $s:\sX\rightarrow V$ the zero section. Then
	\[M\Big(\frac{V}{V\setminus s}\Big)\simeq M(\sX)(d)[2d].\]
\end{corollary}
\begin{proof}
	From the previous lemma we have an isomorphism
	\[M\Big(\frac{V}{V\setminus s}\Big)\simeq M\Big(\frac{\Proj(V\oplus \sO)}{\Proj(V\oplus \sO\setminus s)}\Big),\]
	and we have 
	\[M\Big(\frac{\Proj(V\oplus \sO)}{\Proj(V\oplus \sO\setminus s)}\Big)\simeq M(\sX)(d)[2d]\]
	from the projective bundle formula (see \cite[Lemma 3.9]{MotivesDM} for details).
\end{proof}

\begin{theorem}[Gysin Triangle]
	\label{gysin-triangle}
	Let $\sZ \subset \sX$ be a smooth closed substack of codimension c. Then there exists a Gysin triangle:
	\[M(\sX \setminus \sZ)\rightarrow M(\sX)\rightarrow M(\sZ)(c)[2c] \rightarrow M(\sX \setminus \sZ)[1].\]
\end{theorem}
\begin{proof}
	Note that we have an exact triangle
	\[M(\sX \setminus \sZ)\rightarrow M(\sX)\rightarrow M\Big(\frac{\sX}{\sX\setminus\sZ}\Big) \rightarrow M(\sX \setminus \sZ)[1].\]
	So it suffices to show that $M\big(\frac{\sX}{\sX\setminus\sZ}\big)\simeq M(Z)(c)[2c]$ in $\DM{k,\Z}{}$. The argument is exactly as in \cite[Theorem 3.10]{MotivesDM}.
\end{proof}

\section{Applications to motivic cohomology of algebraic stacks}\label{section-applications-motivic-cohomology}

The significance of Theorem \ref{motive-construction} and other results in this article lies in the fact that they provide a language for us to discuss Chow groups of stacks in manner analogous to the discussion in \cite{MWV} for smooth schemes. For instance, in Theorem \ref{theorem-hypercohomology} we show that the motivic cohomology of cd-quotient stacks can be computed as hypercohomology of motivic complexes. For quotient stacks, Theorem \ref{chow-comparison} implies that the same holds for higher Chow groups. This approach to Chow groups admits further generalisation, and in \cite{khan2021generalized} Chow groups have been connected with invariants in Voevodsky's stable homotopy category.

Once this language connecting Chow groups to hypercohomology is made available, we proceed with the study of cohomology of motivic complexes with $\Z$ and other coefficients. It is natural to ask if analogues of well-known and important theorems hold in this setting. In Corollary \ref{corollary-bl}, we prove a Beilinson-Lichtenbaum type result for stacks. Such a result was previously inaccessible because of the absence of a good notion of Nisnevich motive for stacks.

\subsection{Motivic cohomology as hypercohomology}

For a smooth scheme $X$ over a perfect field, consider the hypercohomology groups $\mathbb{H}^p_{Zar}(X,\Z(q))$ of the motivic complexes $\Z(q)$ in the category complexes of abelian sheaves on the Zariski (or Nisnevich) site of $X$. By \cite[Proposition 14.16]{MWV}, we have the following
\[Hom_{\DM{k,Z}{}}(X,\Z(q)[p])\simeq \mathbb{H}^p_{Zar}(X,\Z(q))\simeq \mathbb{H}^p_{Nis}(X,\Z(q)),\]
thus showing that both Zariski and Nisnevich hypercohomologies of the complexes $\Z(q)$ compute the motivic cohomology groups of $X$.

In this subsection, we establish a similar result relating the motivic cohomology of a cd-quotient stack $\sX$ to the hypercohomology of the motivic complexes on the smooth-Nisnevich site of $\sX$.

\begin{definition}[Smooth-Nisnevich site]
	Let $\sX$ be an algebraic stack. The smooth-Nisnevich site of $\sX$, denoted by $\sX_{\text{lis-nis}}$ is the category whose objects consist of pairs $(U,p)$ where $p$ is a smooth morphism $p:U\rightarrow \sX$  from an algebraic space $U$. Coverings in $\sX_{\text{lis-nis}}$ are given by Nisnevich covering of algebraic spaces.
\end{definition}

Let $DA(\sX)$ denote the category of Morel-Voevodsky motives over $\sX$. This category is constructed analogously as in the case of schemes done in \cite{ayoub-hopf}. The motivic complexes $\Z(j)$ can be thought of as objects in $DA(\sX)$ by restricting along the structure map $\sX\rightarrow\Spec k$.

\begin{theorem}\label{theorem-hypercohomology}
	The motivic cohomology of $\sX$ agrees with the hypercohomology of the motivic complexes $\Z(j)$ on $\sX_{\text{lis-nis}}$. That is,
	\[Ext^i_{D(\sX)}(\Z,\Z(j)|_\sX)\simeq Ext^i_{DA(\sX)}(\Z,\Z(j)|_\sX)\simeq Hom_{\DM{k,\Z}{}}(M(\sX),\Z(j)[i]),\]
	where $\Z$ denotes the constant sheaf $\Z$ on the $\sX_{\text{lis-nis}}$.
\end{theorem}

\begin{proof}
	To prove this result, we will use the existence of $g_\#$ for a map $g: U_{\bullet}\rightarrow \Spec k$ of simplicial sheaves. This theory has been developed in \cite{ayoubthesis2}.
	
	The first equality follows from the fact that $\Z(j)$ are $\A^1$-local complexes.
	
	Let $p: U_{\bullet}\rightarrow\sX$ be a $GL_n$-presentation.  By Lemma \ref{lemma-principal-bundles-nisnevich-sections} the morphism $p: U_{\bullet}\rightarrow \sX$ is a Nisnevich local weak equivalence over $\sX$ and, therefore, $\Z(U_{\bullet})\simeq \Z$ in $DA(\sX)$. Thus, we have
	\[Ext^i_{DA(\sX)}(\Z,\Z(j)|_\sX)= Ext^i_{DA(\sX)}(\Z(U_{\bullet}),\Z(j)|_\sX).\]
	Let $g:U_{\bullet}\rightarrow \Spec k$ and $h: \sX\rightarrow \Spec k$ be the structure maps. Since $\Z(j)$ in $DA(\sX)$ is the pullback of motivic sheaf $\Z(j)$ over $k$, we have,
	
	\begin{align*}
	Ext^i_{DA(\sX)}	(\Z(U_{\bullet}),\Z(j)|_\sX) & \cong Ext^i_{DA(\sX)}(p_\# \Z|_{U_{\bullet}},\Z(j)|_\sX)\\
	&\cong Ext^i_{DA(U_{\bullet})}(\Z|_{U_{\bullet}},g^*\Z(j))\\
	&\cong  Ext^i_{DA(k)}(g_\#(\Z|_{U_{\bullet}}),\Z(j))\\
	&\cong  Ext^i_{DA(k)}(\Z(U_{\bullet}),\Z(j)).
	\end{align*}
	Since $\Z(j)$ are $\A^1$-local complexes, we have that
	\[Ext^i_{DA(k)}(\Z(U_{\bullet}),\Z(j)) \simeq Ext^i_{D(k)}(\Z(U_{\bullet}),\Z(j)) \simeq Hom_{\DM{k,\Z}{}}(M(U_{\bullet}),\Z(j)[i]).\]
	
	Now, the result follows since $U_{\bullet}\rightarrow \sX$ is a local weak equivalence in $\sH(k)$.
\end{proof}

\subsection{Beilinson-Lichtenbaum for stacks}
In this subsection, we establish a Beilinson-Lichtenbaum type result for cd-quotient stacks. While the result is a straightforward corollary of the existing results known for schemes, it requires the existence of a notion of Nisnevich motivic cohomology for stacks. This is guaranteed by Theorem \ref{motive-construction}.

For a cd-quotient stack $\sX$, let $H^{p,q}_M (\sX,\Z/n\Z):= Hom_{\DM{k,\Z}{}}(\sX,\Z/n\Z(q)[p])$ denote the motivic cohomology with $\Z/n\Z$ coefficients. Let $\mu_n$ denote the sheaf of $n$-roots of unity.

\begin{corollary}\label{corollary-bl}
	Let $\sX$ be a cd-quotient stack. Then the homomorphisms
	\[H^{p,q}_M (\sX,\Z/n\Z)\rightarrow H^p_{\acute{e}t}(\sX,\mu_n^{\otimes q}),\]
	are isomorphisms for $p\leq q$ and monomorphisms for $p=q+1$.
\end{corollary}
\begin{proof}
	First note that we have an equivalence, by Theorem \ref{motive-construction},
	\[Hom_{\DM{k,\Z}{}}(\sX,\Z/n\Z(q)[p])= Hom_{\DM{k,\Z}{}}(U_{\bullet},\Z/n\Z(q)[p])\]
	By \cite[Theorem 6.17]{voevodsky-finite-coefficients}, the change of topology morphisms 
	\[Hom_{\DM{k,\Z}{}}(U_{\bullet},\Z/n\Z(q)[p]) \rightarrow Hom_{\DM{k,\Z}{\acute{e}t}}(U_{\bullet},\Z/n\Z(q)[p])\]
	are isomorphisms for $p\leq q$ and monomorphisms for $p=q+1$.\\
	By \cite[Proposition 1.26]{ayoub-hopf}, we can remove transfers on the right hand-side of the above equality to get	
	\[Hom_{\DM{k,\Z}{\acute{e}t}}(U_{\bullet},\Z/n\Z(q)[p])= Hom_{\mathbf{DA}(k)}(U_{\bullet},\mu_n^{\otimes q}[p]).\]
	By \cite[Proposition 1.26]{ayoub-hopf}, this computes \'{e}tale cohomology with $\mu_n^{\otimes q}$-coefficients of the simplicial scheme $U_{\bullet}$. That is,
	\[Hom_{\mathbf{DA}(k)}(U_{\bullet},\mu_n^{\otimes q}[p])\simeq H_{\acute{e}t}^p(U_{\bullet},\mu_n^{\otimes q}).\]
	Finally, again by Theorem \ref{motive-construction}, we have isomorphisms
	\[H_{\acute{e}t}^p(U_{\bullet},\mu_n^{\otimes q})\simeq H_{\acute{e}t}^p(\sX,\mu_n^{\otimes q}).\]
\end{proof}

\subsection{Gersten complex for Deligne-Mumford stacks} Let $\sX$ be a smooth Deligne-Mumford stack. Consider the Zariski site $\sX_{Zar}$ of $\sX$. Let $\sH^{p,q}(-):= Hom(-, \Z(q))$ denote the motivic cohomology sheaves on $\sX_{Zar}$. For a point $x\in |\sX|$, let $\sG_x$ denote the residual gerbe \cite{LMB} at $x$. This is a locally closed substack of $\sX$. Let $i_x:\sG_x\to \sX$ denote the canonical morphism. Then we have a sequence
\begin{equation}\label{equation-gersten-stacks}
0\rightarrow \sH^{p,q}\overset{\partial_0}{\rightarrow} \underset{\codim (\sG_x)=0}{\bigoplus}i_{x*} H^{p,q}(\sG_x) \overset{\partial_1}{\rightarrow} \underset{\codim (\sG_y)=1}{\bigoplus}i_{y*} H^{p-1,q-1}(\sG_y)\overset{\partial_2}{\rightarrow}\ldots
\end{equation}
where we define the constant sheaf $H^{p,q}(\sG_x):=\underset{x\in \sU}{\colim} H^{p,q}(\sU)$, for every smooth open substacks $\sU\subset \sX$ and $x\in \sX$ a generic point. We will now describe what boundary maps are. 

Note the map $\partial_0$ is simply induced by the projections $\sH^{p,q}(\sU)\rightarrow H^{p,q}(\sG_x)$ whenever $x\in \sU$.
To describe the subsequent maps, it suffices to describe the case when $\sX$ is irreducible and $\sG_x\subset \sX$ has codimension one. The groups $H^{p,q}(\sG_x)$ are defined as a limit over smooth neighbourhoods of the point $x$. By shrinking $\sX$, we can assume that $\sX$ is smooth and that we can reduce to a pair $(\sV,\sZ)$, where $\sV\subset \sX$ is an open substack and $\sZ:= \sX\setminus \sV$. Then using the Gysin triangle (Theorem \ref{gysin-triangle}), we have a map,
\[H^{p,q}(\sV)\rightarrow H^{p-1,q-1}(\sZ)\]
which induces the map $\partial_1$.

As with schemes, the sequence (\ref{equation-gersten-stacks}) is a chain complex by construction. We ask the following question which will be addressed in a subsequent work:

\begin{q}
	Is the complex (\ref{equation-gersten-stacks}) exact?
\end{q}

\begin{remark}
	The complex (\ref{equation-gersten-stacks}) is meaningful because we have a notion of the Nisnevich motive by Theorem \ref{motive-construction}. In the theory of higher Chow groups for group actions on schemes, Gersten complexes are constructed using techniques of equivariant localisation. The advantage of the Gersten complex (\ref{equation-gersten-stacks}) is that methods of sheaf cohomology become readily available to us.
\end{remark}

\section{Application to Exhaustive stacks}
\label{section-exhaustive}
In \cite{hoskins2019}, a motive $M_{exh}$ is defined for a class of stacks which they call as \textit{exhaustive stacks}. They do this by using an idea similar to Totaro's ``finite-dimensional approximation" technique in \cite{TotaroChow}. Examples of exhaustive stacks are quotient stacks and the moduli stack of vector bundles on a curve of fixed rank and degree. In fact, exhaustive stacks turn out to be special cases of cd-quotient stacks (see Lemma \ref{proposition-filtration-by-quotients}). We will compare the motive $M_{exh}$ with the motive of Definition \ref{definition-motive-stack}.

In this section, we adopt the conventions used in \cite{hoskins2019} for algebraic stacks. In particular this means that we work with stacks $\sX$ which admit a smooth atlas $p:U\rightarrow \sX$ by a locally finite type $k$-scheme $U$ such that $p$ is schematic (representable by \textit{schemes})\footnote{This is only to maintain consistency with the conventions in \cite{hoskins2019}. It does not particularly affect the arguments that we present, which work for any stack locally of finite type over $k$.}.

\begin{definition}[\!\! {\cite[Definition 2.15]{hoskins2019}}]
	Let $\sX$ be an algebraic stack locally of finite type over a field $k$. Let $\sX_0 \subset \sX_1\subset \ldots$ be an increasing filtration of $\sX$ such $\sX_i\subset \sX$ are quasi-compact open substacks and their union covers $\sX$, i.e, $\sX=\cup_i \sX_i$. Then an \textit{exhaustive sequence of vector bundles} with respect to this filtation is a sequence of pairs $\{(V_i,W_i)\}_{i\geq 0}$ where $V_i$ is a vector bundle on $\sX_i$ and $W_i\subset V_i$ is a closed substack such that
	\begin{enumerate}
		\item the complement $U_i:= V_i\setminus W_i$ is a separated $k$-scheme of finite type,
		\item we have injective maps of vector bundles $f_{i,i+1}:V_i\rightarrow V_{i+1}\times_{\sX_{i+1}}\sX_i$ such that $f_{i,i+1}^{-1}(W_{i+1}\times_{\sX_{i+1}}\sX_i)\subset W_i$ and,
		\item the codimension of $W_i$ in $V_i$ tends to infinity as $i$ increases.
	\end{enumerate}
	A stack admitting an exhaustive sequence with respect to some filtration is said to be \textit{exhaustive}.
\end{definition}

In fact, one can show that every exhaustive stack admits a filtration by global quotient stacks. This implies that it is a cd-quotient stack.

\begin{lemma}
	\label{proposition-filtration-by-quotients}
	Let $\sX$ be an exhaustive stack. Let $\sX=\cup_i\sX_i$  be an increasing filtration with an exhaustive sequence of vector bundles. Then there exists an increasing filtration $\sX=\cup_i\sY_i$ with $\sY_i\subseteq \sX_i$ and each $\sY_i$ is a global quotient stack. In particular, it is a cd-quotient stack.
\end{lemma}

\begin{proof}
	Let $\{(V_{i},W_i)\}_{i\geq 0}$ denote the exhaustive sequence of vector bundles corresponding to the filtration $\{\sX_i\}_{i\geq 0}$. Let $p_i:V_i\rightarrow \sX_i$ denote the structure map of the vector bundle $V_i$. Now by definition the complement $U_i= V_i\setminus W_i$ is a separated finite type $k$-scheme. Since $p_i$ is smooth, the image $p_i(U_i)\subset \sX_i$ is an open substack which is of finite type over $k$. Set $\sY_i:=p_i(U_i)$. Consider the restriction $V_i\times_{\sX_i}\sY_i\rightarrow \sY_i$ of $V_i$ to this substack. This is a vector bundle on $\sY_i$ which contains an open representable substack $U_i\subset V_i\times_{\sX_i}\sY_i$ that surjects onto $\sY_i$. Thus, $\sY_i$ is a global quotient stack, by \cite[Lemma 2.12]{EHKV}.
	
	Thus, we have an increasing filtration $\{\sY_i\}_{i\geq 0}$ such that $\sY_i\subseteq \sX_i$. The only thing left to check is that this filtration covers $\sX$. This follows from the following topological argument.
	
	Take a point $x\in \sX$. Then as the filtration $\{\sX_i\}_{i\geq 0}$ covers $\sX$, there exists an $i$ such that $x\in \sX_i$. We will show that there exists an $N\geq i$ such that $x\in \sY_N$. 
	
	If $x\in \sY_i$, there is nothing to prove. So assume that $x\notin \sY_i$. This means that $p_i^{-1}(x)\subset W_i$ in the vector bundle $V_i$. Let $Z:=\overline{p_i^{-1}\{x\}}$ be the closure of the fibre in $V_i$. Since $Z\subseteq W_i$ we see that 
	\[n:=\codim\, Z\geq \codim W_i.\]
	As $\{(V_i,W_i)\}_{i\geq 0}$ is an exhaustive sequence, there exists an $N$ such that $\codim\, W_N > n$. Further, we have a map $f_{i,N}:V_i\rightarrow V_N$ such that $f_{i,N}^{-1}(W_N)\subset W_i$. If $Z$ was contained in $f_{i,N}^{-1}(W_N)$, we would have
	\[n=\codim\, Z\geq \codim \, f_{i,N}^{-1}(W_N)> n,\]
	a contradiction. Thus, there exists $y\in p_i^{-1}(x)$ such that $f_{i,N}(y)\in U_N$ implying that $x\in \sY_N$.
\end{proof}

\begin{definition}[\!\!{\cite[Definition 2.17]{hoskins2019}}]
Let $\sX$ be an exhaustive stack with an exhaustive sequence of vector bundles $\{(V_i,W_i)\}_{i\geq 0}$. The motive $M_{exh}(\sX)$ is defined in as the colimit of the motives of the schemes $U_i$. That is,
\[M_{exh}(\sX)=\colim M(U_i).\]
\end{definition}

Since exhaustive stacks are cd-quotient stacks, we would like to compare the motive $M_{exh}(\sX)$ with the motive $M(\sX)$ in Definition \ref{definition-motive-stack}. The following proposition shows that they are isomorphic in $\DM{k,\Z}{}$.

\begin{proposition}
	\label{proposition-comparision-with-hoskins19}
	Let $\sX$ be a smooth exhaustive stack. Then
	\[M(\sX) \simeq M_{exh}(\sX) \;\; \text{in}\;\; \DM{k,\Z}{}\]
\end{proposition}
\begin{proof}
	Let $X\rightarrow \sX$ be the $0$-skeleton of a $GL_n$-presentation. By \cite[Theorem II.6.4]{knutson}, there exists a Nisnevich covering $Y\rightarrow X$ with $Y$ a scheme. This gives us a presentation $Y\rightarrow \sX$. Let $Y_{\bullet}\rightarrow \sX$ be the associated \v{C}ech nerve. By similar argument as in Theorem \ref{motive-construction}, we get that $Y_{\bullet} \simeq \sX$ in $\mathcal{H}_{\bullet}(k)$, and so we have $M(Y_{\bullet}) \simeq M(\sX)$ in $\DM{k,\Z}{}$. Hence, it suffices to show that $M(Y_{\bullet}) \simeq M_{exh}(\sX)$ in $\DM{k,\Z}{}$.	The proof is exactly the same as \cite[Proposition A.7]{hoskins2019} using the atlas $Y_{\bullet}\rightarrow\sX$.
\end{proof}

\begin{example}
	Let $C$ be a smooth projective geometrically connected curve of genus $g$ over a field $k$.	In \cite[Section 3]{hoskins2019}, it is proved that the moduli stack $Bun_{n,d}$ of vector bundles on a $C$ of fixed rank $n$ and degree $d$ is exhaustive. To show this, they observe that it admits a filtration by the maximal slope of all vector bundles. Thus, $Bun_{n,d}$ is a filtered colimit of open substacks $Bun_{n,d}^{\geq \mu_l}$ where $\{\mu_l\}$ is an increasing sequence of rational number representing the maximal slope. Also, each of these open substacks is a global quotient stack and can be written as $Bun_{n,d}^{\geq \mu_l}:=[Q^{\geq \mu_l}/GL_N]$ where $Q^{\geq \mu_l}$ is an open subscheme of a Quot scheme (see \cite[Th\'{e}or\`{e}me 4.6.2.1]{LMB} for futher details). Then, by \cite[Lemma 2.26]{hoskins2019}, we have
	\[M_{exh}(Bun_{n,d})=\hocolim_l M_{exh}(Bun_{n,d}^{\geq \mu_l})\]
	and from Proposition \ref{proposition-comparision-with-hoskins19} we get that
	\[M(Bun_{n,d})=\hocolim_l M(Bun_{n,d}^{\geq \mu_l})=\hocolim_l M(Q^{\geq \mu_l}_{\bullet}).\]
	Thus, for $Bun_{n,d}$, the motive $M_{exh}$ of \cite{hoskins2019} can be computed as a homotopy colimit of the motive of Definition \ref{definition-motive-stack}.
	Since these homotopy colimits are being taken over filtered categories, they can actually be computed by their ordinary colimits (see \cite[Example 12.3.5]{BK}).
\end{example}

\begin{remark}
	Proposition \ref{proposition-comparision-with-hoskins19} shows that for exhaustive stacks the motive defined using the nerve construction in Definition \ref{definition-motive-stack} agrees with $M_{exh}$ defined in \cite{hoskins2019} in the Nisnevich topology. This potentially simplifies many of the functoriality arguments in \cite[\S 2]{hoskins2019}. For example, it is now immediate that $M_{exh}$ is independent of the choices involved in its construction (see also \cite[Lemma 2.20]{hoskins2019}). In \cite[Appendix A]{hoskins2019}, such a comparison is proved for \'{e}tale motives.
\end{remark}

\appendix
\section{Homotopy (co)limits and derived functors}

Let $I$ be an index category, and let $M^I$ denote the category of $I$-diagrams in $M$. $M^I$ is just the cateogry of functors $Fun(I,M)$. We have a constant functor $c:M\rightarrow M^I$, taking every object $A$ of $M$ to the constant $I$-diagram whose every object is $A$ and all maps are $id_A$. The left adjoint of $c$ (if it exists) is the colimit functor $\colim :M^I\rightarrow M$.\\
Let $F:M\rightleftarrows N: G$ be an adjunction. Then we also have an induced adjunction between the category of $I$-diagrams,
\[F^I:M^I\rightleftarrows N^I:G.\]

\begin{remark}\label{remark-reference-homotopy-colimits}
	While trying to prove the Projective Bundle formula (Theorem \ref{projective-bundle-formula}), we came across the following lemma. As pointed out to us by the referee this is \cite[Theorem 19.4.5(1)]{hirschhorn}. We include it here for the sake of completeness and thank the referee for pointing us to a reference.
\end{remark}

The following lemma shows that just as left adjoints commute with colimits, derived functors of left adjoints commute with \textit{homotopy colimits}.

\begin{lemma}\label{homotopy-colimits-left-adjoints}
	Let $F:M\rightleftarrows N: G$ be a Quillen adjunction of model categories. Let $I$ be an index category such that the projective model structure is defined on $M^I$ and $N^I$. Then, for an $I$-diagram $E$ in $M$, we have
	\[LF(\hocolim E)\simeq \hocolim LF^I(E),\]
	where $LF$ and $LF^I$ are the derived functors of $F$ and $F^I$, repectively.
\end{lemma}
\begin{proof}
	Note that $LF$ is defined as the composite
	\[Ho(M)\overset{Q}{\rightarrow} Ho(M_c)\overset{F}{\rightarrow} Ho(N)\]
	where $Q$ is the cofibrant replacement functor.
	
	Let $M^I$ denote the category of $I$-diagrams in $M$. Note that the fibrations and weak equivalences in the projective model structure on $M^I$ are defined object-wise. The homotopy colimit functor is the derived functor of the colimit functor $\colim : M^I\rightarrow M$ which is the left adjoint of the constant functor $M\rightarrow M^I$. More precisely,
	\[\hocolim: Ho(M^I)\overset{Q}{\rightarrow}Ho(M^I)\overset{\colim}{\rightarrow}Ho(M).\]
	Note that here $Q$ is the cofibrant replacement functor in the \textit{projective model structure} of $M^I$.
	This says that for any $I$-diagram $E$ the homotopy colimit can be computed by taking the ordinary colimit of its cofibrant replacement $QE$, i.e,
	\[\hocolim E\simeq \colim QE\]
	is a weak equivalence in the homotopy category.
	Since, $M^I$ has the projective model structure, {\rm colim} is left Quillen. Thus, $\colim QE$ is, in fact, a cofibrant object in $M$ and
	\begin{align*}
	LF(\hocolim E) &\simeq LF(\colim QE) \\
	&\simeq F(\colim QE).
	\end{align*}
	
	Now, observe that the adjoint pair $(F,G)$ induces an adjunction on the diagram categories associated to $I$,
	\[F^I:M^I\rightleftarrows N^I:G^I.\]
	As fibrations are defined object-wise and $G$ is right Quillen, $G^I$ preserves fibrations. Hence, $F^I$ is left Quillen, and preserves cofibrations. This means that the image $F^I(QE)$ is cofibrant in $N^I$.
	Then, we have
	\begin{align*}
	\hocolim LF^I(E) &\simeq \hocolim F^I(QE)\\
	&\simeq \colim F^I(QE).
	\end{align*}

As $F$ is a left adjoint, it commutes with ordinary colimits. That is, we have a commutative diagram,
\begin{center}
	\begin{tikzcd}
	M^I\arrow[r,"F^I"]\arrow[d] & N^I\arrow[d]\\
	M\arrow[r,"F"] &N
	\end{tikzcd}
\end{center}
where the vertical arrows are the colimit functor. Thus,
\[F(\colim QE)=\colim F^I(QE),\]
as required.
\end{proof}

\bibliography{mybib.bib}
\bibliographystyle{alphanum}

\end{document}